\documentclass[oneside,reqno,american,english]{amsart}
\usepackage[latin9]{inputenc}
\usepackage{amstext}
\usepackage{amsthm}
\usepackage{amssymb}
\usepackage{geometry}
\usepackage{setspace}

\makeatletter
\theoremstyle{plain}
\newtheorem{thm}{\protect\theoremname}

\@ifundefined{date}{}{\date{}}
\PassOptionsToPackage{reqno}{amsmath}

\usepackage{mathtools}

\makeatletter
\AtBeginDocument{\tagsleft@false} % ensures numbers appear on the right
\makeatother

\usepackage{chngcntr}
\counterwithout{equation}{section}

\usepackage{microtype}   % improves text spacing and kerning
\usepackage{setspace}    % control line spacing
\makeatother

\usepackage{babel}
\addto\captionsamerican{\renewcommand{\theoremname}{Theorem}}
\addto\captionsenglish{\renewcommand{\theoremname}{Theorem}}
\providecommand{\theoremname}{Theorem}

\begin{document}
\title{``Infinitely Often'' Transcendence of Gamma-Function Derivatives}
\author{Michael R. Powers}
\address{Department of Finance, School of Economics and Management, Tsinghua
University, Beijing, China 100084}
\email{powers@sem.tsinghua.edu.cn}
\date{April 21, 2026}
\begin{abstract}
Relatively little is known about the arithmetic properties of Gamma-function
derivatives evaluated at arbitrary points $q\in\mathbb{Q}\setminus\mathbb{Z}_{\leq0}$.
In recent work, we showed that the sequence $\left\{ \Gamma^{\left(n\right)}\left(1\right)\right\} _{n\geq1}$
contains transcendental elements infinitely often. That result is
now generalized to all sequences $\left\{ \Gamma^{\left(n\right)}\left(q\right)\right\} _{n\geq1}$
for $q\in\tfrac{1}{2}\mathbb{Z}\setminus\mathbb{Z}_{\leq0}$. Moreover,
for all such $q$ we derive a lower bound, $\beta\left(N\right)=\max\left\{ 0,\sqrt{N}-5/2\right\} /N$,
for the density of transcendental elements $\Gamma^{\left(n\right)}\left(q\right)$
among $n\in\left\{ 1,2,\ldots,N\right\} $, where $\beta\left(N\right)\asymp N^{-1/2}\rightarrow0$
as $N\rightarrow\infty$. For $q\in\mathbb{Q}\setminus\tfrac{1}{2}\mathbb{Z}$,
we find the somewhat weaker result that at least one of the sequences
$\left\{ \Gamma^{\left(n\right)}\left(q\right)\right\} _{n\geq1}$,
$\left\{ \Gamma^{\left(n\right)}\left(1-q\right)\right\} _{n\geq1}$
contains infinitely many transcendental elements.
\end{abstract}

\keywords{Gamma-function derivatives; Euler-Mascheroni constant; transcendence;
density bounds.}
\maketitle

\section{Introduction}

\begin{singlespace}
\noindent For any $q\in\mathbb{Q}\setminus\mathbb{Z}_{\leq0}$, the
Taylor-series expansion 
\begin{equation}
\Gamma\left(q+t\right)=\sum_{n=0}^{\infty}\Gamma^{\left(n\right)}\left(q\right)\frac{t^{n}}{n!}
\end{equation}
(for sufficiently small $\left|t\right|$) encodes the Gamma-function
derivatives $\Gamma^{\left(n\right)}(q)$ as coefficients. Although
$\Gamma\left(z\right)$ is well understood analytically for $z\in\mathbb{R}\setminus\mathbb{Z}_{\leq0}$
(see, e.g., Artin {[}1{]}), it appears that relatively little is known
unconditionally about the arithmetic nature of $\Gamma^{\left(n\right)}\left(q\right)$
at rational points $q$ (see Rivoal {[}2, 3{]}).
\end{singlespace}

For the cases of $q\in\left\{ 1,\tfrac{1}{2}\right\} $, the sequence
$\left\{ \Gamma^{\left(n\right)}\left(q\right)\right\} _{n\geq1}$
satisfies the recursions
\[
\Gamma^{\left(n\right)}\left(1\right)=-\gamma\Gamma^{\left(n-1\right)}\left(1\right)+{\displaystyle \sum_{k=1}^{n-1}}\dfrac{\left(-1\right)^{k+1}\left(n-1\right)!}{\left(n-k-1\right)!}\zeta\left(k+1\right)\Gamma^{\left(n-k-1\right)}\left(1\right),
\]
and
\[
\Gamma^{\left(n\right)}\left(\tfrac{1}{2}\right)=-g\Gamma^{\left(n-1\right)}\left(\tfrac{1}{2}\right)+{\displaystyle \sum_{k=1}^{n-1}}\dfrac{\left(-1\right)^{k+1}\left(n-1\right)!}{\left(n-k-1\right)!}\left(2^{k+1}-1\right)\zeta\left(k+1\right)\Gamma^{\left(n-k-1\right)}\left(\tfrac{1}{2}\right),
\]
respectively, where $\gamma=0.577215\ldots$ denotes the Euler-Mascheroni
constant, $g=\gamma+2\ln\left(2\right)$, and $\Gamma^{\left(0\right)}\left(q\right)\equiv\Gamma\left(q\right)$.
These identities can be iterated to yield the systems\footnote{See, e.g., Choi and Srivastava {[}4{]}.}
\begin{equation}
\begin{array}{c}
\Gamma^{\left(1\right)}\left(1\right)=-\gamma,\\
\Gamma^{\left(2\right)}\left(1\right)=\gamma^{2}+\zeta\left(2\right),\\
\Gamma^{\left(3\right)}\left(1\right)=-\gamma^{3}-3\zeta\left(2\right)\gamma-2\zeta\left(3\right),\\
\Gamma^{\left(4\right)}\left(1\right)=\gamma^{4}+6\zeta\left(2\right)\gamma^{2}+8\zeta\left(3\right)\gamma+\dfrac{27}{2}\zeta\left(4\right),\\
\Gamma^{\left(5\right)}\left(1\right)=-\gamma^{5}-10\zeta\left(2\right)\gamma^{3}-20\zeta\left(3\right)\gamma^{2}-\dfrac{135}{2}\zeta\left(4\right)\gamma-20\zeta\left(2\right)\zeta\left(3\right)-24\zeta\left(5\right),\\
\cdots
\end{array}
\end{equation}

\begin{singlespace}
\noindent and\pagebreak
\begin{equation}
\begin{array}{c}
\Gamma^{\left(1\right)}\left(\tfrac{1}{2}\right)=-\sqrt{\pi}g,\\
\Gamma^{\left(2\right)}\left(\tfrac{1}{2}\right)=\sqrt{\pi}\left(g^{2}+3\zeta\left(2\right)\right),\\
\Gamma^{\left(3\right)}\left(\tfrac{1}{2}\right)=-\sqrt{\pi}\left(g^{3}+9\zeta\left(2\right)g+14\zeta\left(3\right)\right),\\
\Gamma^{\left(4\right)}\left(\tfrac{1}{2}\right)=\sqrt{\pi}\left(g^{4}+18\zeta\left(2\right)g^{2}+56\zeta\left(3\right)g+\dfrac{315}{2}\zeta\left(4\right)\right),\\
\Gamma^{\left(5\right)}\left(\tfrac{1}{2}\right)=-\sqrt{\pi}\left(g^{5}+30\zeta\left(2\right)g^{3}+140\zeta\left(3\right)g^{2}+\dfrac{1575}{2}\zeta\left(4\right)g+420\zeta\left(2\right)\zeta\left(3\right)+744\zeta\left(5\right)\right),\\
\cdots,
\end{array}
\end{equation}
where: (i) each $\zeta\left(n\right),\:n\in\left\{ 2,4,\ldots\right\} $
is a rational multiple of $\pi^{n}$; (ii) $\zeta\left(3\right)$
is known to be irrational (but not necessarily transcendental);\footnote{See Apéry {[}5{]} and Beukers {[}6{]}.}
and (iii) the arithmetic properties of $\gamma$ and $\zeta\left(n\right),\:n\in\left\{ 5,7,\ldots\right\} $
remain unsettled.\footnote{See Lagarias {[}7{]}, Rivoal {[}2, 3{]}, Ball and Rivoal {[}8{]},
and Fischler, Sprang, and Zudilin {[}9{]}.}
\end{singlespace}

\begin{singlespace}
Although the polynomial forms on the right-hand sides of (2) and (3)
are generally thought to yield transcendental numbers, no specific
value of $\Gamma^{\left(n\right)}\left(1\right)$ or $\Gamma^{\left(n\right)}\left(\tfrac{1}{2}\right)$,
for $n\in\mathbb{Z}_{\geq1}$, is currently known to be transcendental
(or even irrational). In fact, we are aware of only two \emph{unconditional}
transcendence results for these systems:

${\footnotesize {\small \bullet}\:}$ at least one element of each
pair $\left\{ \Gamma^{\left(1\right)}\left(1\right),\Gamma^{\left(2\right)}\left(1\right)\right\} $,
$\left\{ \Gamma^{\left(1\right)}\left(\tfrac{1}{2}\right),\Gamma^{\left(2\right)}\left(\tfrac{1}{2}\right)\right\} $
is transcendental;\footnote{For the former pair, we have $\Gamma^{\left(2\right)}\left(1\right)-\left(\Gamma^{\left(1\right)}\left(1\right)\right)^{2}=\zeta\left(2\right)=\pi^{2}/6$,
but the difference of two algebraic numbers cannot be transcendental.
For the latter pair, we have $\sqrt{\pi}\Gamma^{\left(2\right)}\left(\tfrac{1}{2}\right)-\left(\Gamma^{\left(1\right)}\left(\tfrac{1}{2}\right)\right)^{2}=3\pi\zeta\left(2\right)=\pi^{3}/2$,
which implies $x=\sqrt{\pi}$ is a root of $x^{6}-2\Gamma^{\left(2\right)}\left(\tfrac{1}{2}\right)x+2\left(\Gamma^{\left(1\right)}\left(\tfrac{1}{2}\right)\right)^{2}=0$;
however, the root of a polynomial equation with algebraic coefficients
cannot be transcendental.} and

${\footnotesize {\small \bullet}\:}$ the values $\Gamma^{\left(n\right)}\left(1\right)$
are transcendental infinitely often (as recently proved by Powers
{[}10{]}).\footnote{Fischler and Rivoal {[}11{]} provided a deep conditional transcendence
result, showing that the $\Gamma^{\left(n\right)}\left(q\right)$
are transcendental for all $q\in\mathbb{Q}_{>0}$ if the unproved
conjecture $\mathbf{E}\cap\mathbf{D}=\overline{\mathbb{Q}}$ is true
(where $\mathbf{E}$ and $\mathbf{D}$ are rings of values taken at
algebraic points by arithmetic Gevrey series of orders $-1$ and $1$,
respectively). See also Fischler and Rivoal {[}12{]} and André {[}13{]}.}\smallskip{}

In the present work, we investigate the ``infinitely often'' transcendence
of $\Gamma^{\left(n\right)}\left(q\right)$ for $q\in\mathbb{Q}\setminus\mathbb{Z}_{\leq0}$.
First, Theorem 1 demonstrates that for every positive integer and
every half-integer (i.e., for all $q\in\tfrac{1}{2}\mathbb{Z}\setminus\mathbb{Z}_{\leq0}$),
the sequence $\left\{ \Gamma^{\left(n\right)}\left(q\right)\right\} _{n\geq1}$
contains infinitely many transcendental elements. This analysis leads
directly to Theorem 2, which shows that for all such $q$ the density
of transcendental elements $\Gamma^{\left(n\right)}\left(q\right)$
among $n\in\left\{ 1,2,\ldots,N\right\} $ is bounded below by
\[
\beta\left(N\right)=\dfrac{\max\left\{ 0,\sqrt{N}-5/2\right\} }{N},
\]
where $\beta\left(N\right)\asymp N^{-1/2}\rightarrow0$ as $N\rightarrow\infty$.
Finally, for $q\in\mathbb{Q}\setminus\tfrac{1}{2}\mathbb{Z}$, Theorem
3 provides a weaker (disjunctive) counterpart to Theorem 1, asserting
that at least one of the sequences $\left\{ \Gamma^{\left(n\right)}\left(q\right)\right\} _{n\geq1}$,
$\left\{ \Gamma^{\left(n\right)}\left(1-q\right)\right\} _{n\geq1}$
contains infinitely many transcendental elements.
\end{singlespace}

\section{Values of $q\in\tfrac{1}{2}\mathbb{Z}\setminus\mathbb{Z}_{\protect\leq0}$}

\subsection{``Infinitely Often'' Transcendental Derivatives}

\begin{singlespace}
\phantom{}

\medskip{}

\end{singlespace}

\begin{singlespace}
\noindent We begin by showing that for all $q\in\tfrac{1}{2}\mathbb{Z}\setminus\mathbb{Z}_{\leq0}$,
transcendental elements appear infinitely often in the sequence $\left\{ \Gamma^{\left(n\right)}\left(q\right)\right\} _{n\geq1}$.
This is proved by arguments based on classical identities for $\Gamma\left(t\right)$
along with constraints imposed by the transcendence of $\pi$.
\end{singlespace}

\begin{singlespace}
Essentially, repeated application of the Gamma functional equation
reduces $\Gamma\left(q\pm t\right)$ to either $\Gamma\left(1\pm t\right)$
or $\Gamma\left(\tfrac{1}{2}\pm t\right)$, and Euler's reflection
formula, in conjunction with the functional equation, gives closed
forms involving $\sin(\pi t)$ or $\cos(\pi t)$, respectively. The
resulting Taylor series then possess coefficients that are algebraic
linear combinations of distinct powers of $\pi$, each with a non-vanishing
highest-degree term. Comparing these coefficients with those obtained
by products of the Taylor series in (1) -- under the hypothesis that
only finitely many $\Gamma^{\left(n\right)}\left(q\right)$ are transcendental
-- forces infinitely many distinct powers of $\pi$ to lie in a fixed
finite-dimensional vector space over the algebraic numbers, thereby
yielding a contradiction.\pagebreak{}
\end{singlespace}
\begin{thm}
For all $q\in\tfrac{1}{2}\mathbb{Z}\setminus\mathbb{Z}_{\leq0}$,
elements of the sequence $\left\{ \Gamma^{\left(n\right)}\left(q\right)\right\} _{n\geq1}$
are transcendental infinitely often.
\end{thm}
\begin{proof}
\begin{singlespace}
\phantom{}
\end{singlespace}

\begin{singlespace}
\medskip{}
\noindent\foreignlanguage{american}{Let $q\in\frac{1}{2}\mathbb{Z}\setminus\mathbb{Z}_{\leq0}$,
and choose $m\in\mathbb{Z}$ and $\xi\in\left\{ 1,\tfrac{1}{2}\right\} $
so that $q=\xi+m$. Repeated application of the functional equation
$\Gamma\left(z+1\right)=z\,\Gamma\left(z\right)$ generates a rational
function $R_{q}\left(t^{2}\right)\in\overline{\mathbb{Q}}\left(t^{2}\right)$,
analytic at $t=0$ and satisfying $R_{q}\left(0\right)\in\overline{\mathbb{Q}}^{\times}$,
such that
\begin{equation}
\Gamma\left(q+t\right)\Gamma\left(q-t\right)=R_{q}\left(t^{2}\right)\Gamma\left(\xi+t\right)\Gamma\left(\xi-t\right).
\end{equation}
(For example, if $m\geq0$ then we can take $R_{q}\left(t^{2}\right)=\prod_{i=0}^{m-1}\left[\left(\xi+i\right)^{2}-t^{2}\right]$,
whereas if $m<0$ then we can use $R_{q}\left(t^{2}\right)=\prod_{i=m}^{-1}\left[\left(\xi+i\right)^{2}-t^{2}\right]^{-1}$.
In either case, $R_{q}\left(0\right)\neq0$ because $q\notin\mathbb{Z}_{\leq0}$.)}
\end{singlespace}

\selectlanguage{american}%
\begin{singlespace}
Now note that Euler's reflection formula, in conjunction with the
functional equation, provides the identities
\[
\Gamma\left(1+t\right)\Gamma\left(1-t\right)=\frac{\pi t}{\sin\left(\pi t\right)}
\]
and
\[
\Gamma\left(\tfrac{1}{2}+t\right)\Gamma\left(\tfrac{1}{2}-t\right)=\frac{\pi}{\cos\left(\pi t\right)},
\]
which together imply that
\begin{equation}
\Gamma\left(\xi+t\right)\Gamma\left(\xi-t\right)={\displaystyle \sum_{j=0}^{\infty}}a_{j,\xi}\pi^{2\left(j+1-\xi\right)}t^{2j},
\end{equation}
with $a_{j,\xi}\in\mathbb{Q}^{\times}$ for all $j\geq0$ (because
the Taylor-series expansions of $\pi t/\sin\left(\pi t\right)$ and
$\pi/\cos\left(\pi t\right)$ contain only even powers with non-zero
rational coefficients -- given explicitly in terms of Bernoulli and
Euler numbers, respectively). Re-expressing $R_{q}\left(t^{2}\right)$
as
\begin{equation}
R_{q}\left(t^{2}\right)={\displaystyle \sum_{k=0}^{\infty}}b_{k,q}t^{2k},
\end{equation}
with $b_{k,q}\in\overline{\mathbb{Q}}$ for all $k\geq0$ and $b_{0,q}=R_{q}\left(0\right)\neq0$,
we then see from (4), (5), and (6) that
\[
\Gamma\left(q+t\right)\Gamma\left(q-t\right)={\displaystyle \sum_{j=0}^{\infty}}c_{j,q}t^{2j},
\]
where
\begin{equation}
c_{j,q}={\displaystyle \sum_{k=0}^{j}}b_{k,q}a_{j-k,\xi}\pi^{2\left(j-k+1-\xi\right)}.
\end{equation}
In particular, for each $j\geq0$ the coefficient $c_{j,q}$ is a
$\overline{\mathbb{Q}}$-linear combination of $\left\{ \pi^{2\left(1-\xi\right)},\pi^{2\left(2-\xi\right)},\ldots,\pi^{2\left(j+1-\xi\right)}\right\} $
whose highest-degree term is $b_{0,q}a_{j,\xi}\pi^{2\left(j+1-\xi\right)}$,
with $b_{0,q}a_{j,\xi}\in\overline{\mathbb{Q}}^{\times}$.

Alternatively, from the Taylor-series expansions of $\Gamma\left(q+t\right)$
and $\Gamma\left(q-t\right)$ at $t=0$, we find
\[
\Gamma\left(q+t\right)\Gamma\left(q-t\right)=\left({\displaystyle \sum_{\ell_{1}=0}^{\infty}}\Gamma^{\left(\ell_{1}\right)}\left(q\right)\frac{t^{\ell_{1}}}{\ell_{1}!}\right)\left({\displaystyle \sum_{\ell_{2}=0}^{\infty}}\Gamma^{\left(\ell_{2}\right)}\left(q\right)\frac{\left(-t\right)^{\ell_{2}}}{\ell_{2}!}\right)
\]
\[
={\displaystyle \sum_{j=0}^{\infty}}\left({\displaystyle \sum_{i=0}^{2j}}\frac{\left(-1\right)^{i}\Gamma^{\left(i\right)}\left(q\right)\Gamma^{\left(2j-i\right)}\left(q\right)}{i!\left(2j-i\right)!}\right)t^{2j},
\]
implying
\begin{equation}
c_{j,q}={\displaystyle \sum_{i=0}^{2j}}\frac{\left(-1\right)^{i}\Gamma^{\left(i\right)}\left(q\right)\Gamma^{\left(2j-i\right)}\left(q\right)}{i!\left(2j-i\right)!}
\end{equation}
for all $j\geq0$.

Now suppose (for purposes of contradiction) that only finitely many
$\Gamma^{\left(n\right)}\left(q\right)$ are transcendental, and choose
$n^{*}\in\mathbb{Z}_{\geq1}$ such that $\Gamma^{\left(n\right)}\left(q\right)\in\overline{\mathbb{Q}}$
for all $n>n^{*}$, with
\[
\mathcal{V}=\textrm{span}_{\overline{\mathbb{Q}}}\left\{ 1,\Gamma\left(q\right),\Gamma^{\left(1\right)}\left(q\right),\Gamma^{\left(2\right)}\left(q\right),\ldots,\Gamma^{\left(n^{*}\right)}\left(q\right)\right\} ,
\]
denoting a finite-dimensional $\overline{\mathbb{Q}}$-vector space.
For any $j>n^{*}$, each summand in (8) possesses at least one algebraic
factor $\Gamma^{\left(i_{j}\right)}\left(q\right)$ for some $i_{j}>n^{*}$
(because $i+\left(2j-i\right)=2j>2n^{*}\Longrightarrow\left(i>n^{*}\right)\vee\left(2j-i>n^{*}\right)$),
and therefore lies in $\mathcal{V}$. Thus, $c_{j,q}\in\mathcal{V}$
for all $j>n^{*}$.

At the same time, it follows from (7) that the family $\left\{ c_{j,q}\right\} _{j>n^{*}}$
is $\overline{\mathbb{Q}}$-linearly independent. This is because:
(i) if there exists a non-trivial $\overline{\mathbb{Q}}$-linear
dependence $\lambda_{1}c_{j_{1},q}+\lambda_{2}c_{j_{2},q}+\cdots+\lambda_{\mu}c_{j_{\mu},q}=0$
(with $j_{1}<j_{2}<\cdots<j_{\mu}$) for arbitrary $\mu\in\mathbb{Z}_{\geq2}$,
then the term $\pi^{2\left(j_{\mu}+1-\xi\right)}$ appears only in
$c_{j_{\mu},q}$ (with non-zero algebraic coefficient $b_{0,q}a_{j_{\mu},\xi}$),
forcing $\lambda_{\mu}=0$ (because otherwise $\pi$ would be algebraic);
and (ii) iterating downward similarly forces $\lambda_{\mu-1}=\lambda_{\mu-2}=\cdots=\lambda_{1}=0$.

Consequently, for arbitrarily large $\mu$, the $\mu$ elements $\left\{ c_{n^{*}+1,q},c_{n^{*}+2,q},\ldots,c_{n^{*}+\mu,q}\right\} $
must be $\overline{\mathbb{Q}}$-linearly independent and lie in $\mathcal{V}$,
which is impossible because $\mathcal{V}$ is finite-dimensional.
This contradiction confirms that $\Gamma^{\left(n\right)}\left(q\right)$
is transcendental infinitely often.
\end{singlespace}
\end{proof}

\subsection{Lower Bound for Transcendental Densities}

\begin{singlespace}
\phantom{}

\medskip{}

\end{singlespace}

\begin{singlespace}
\noindent Let
\[
\delta_{q}\left(N\right)=\dfrac{\#\left\{ 1\leq n\leq N:\Gamma^{\left(n\right)}\left(q\right)\textrm{ is transcendental}\right\} }{N},\quad N\in\mathbb{Z}_{\geq1}
\]
denote the density of transcendental elements $\Gamma^{\left(n\right)}\left(q\right)$
among $n\in\left\{ 1,2,\ldots,N\right\} $, with $q\in\frac{1}{2}\mathbb{Z}\setminus\mathbb{Z}_{\leq0}$
fixed. We now provide a lower bound for this density.
\end{singlespace}
\begin{thm}
For any $q\in\frac{1}{2}\mathbb{Z}\setminus\mathbb{Z}_{\leq0}$,
\[
\delta_{q}\left(N\right)\geq\dfrac{\max\left\{ 0,\left\lceil \sqrt{2\left\lfloor N/2\right\rfloor +9/4}-5/2\right\rceil \right\} }{N}\geq\dfrac{\max\left\{ 0,\sqrt{N}-5/2\right\} }{N}=\beta\left(N\right),
\]
where $\beta\left(N\right)\asymp N^{-1/2}$ as $N\rightarrow\infty$.
\end{thm}
\begin{proof}
\begin{singlespace}
\phantom{}
\end{singlespace}

\begin{singlespace}
\medskip{}
\noindent Fixing $q\in\frac{1}{2}\mathbb{Z}\setminus\mathbb{Z}_{\leq0}$ and
$N\in\mathbb{Z}_{\geq1}$, consider the set of coefficients $\left\{ c_{j,q}\right\} $
for $j\in\left\{ 0,1,\ldots,\left\lfloor N/2\right\rfloor \right\} $.
From (8), we know that each $c_{j,q}$ is a $\overline{\mathbb{Q}}$-linear
combination of products $\Gamma^{\left(u\right)}\left(q\right)\Gamma^{\left(v\right)}\left(q\right)$,
with $0\leq u,v\le N$.
\end{singlespace}

Assume that the number of transcendental elements $\Gamma^{\left(n\right)}\left(q\right)$
among $n\in\left\{ 1,2,\ldots,N\right\} $ is given by $N\delta_{q}\left(N\right)=\tau$
for some integer $\tau\geq0$, and let $\alpha_{1},\alpha_{2},\dots,\alpha_{\tau^{*}}$
denote the $\tau^{*}\leq\tau+1$ distinct transcendental numbers within
the set $\left\{ \Gamma\left(q\right),\Gamma^{\left(1\right)}\left(q\right),\Gamma^{\left(2\right)}\left(q\right),\ldots,\Gamma^{\left(N\right)}\left(q\right)\right\} $,
so that every element of this set lies in $\overline{\mathbb{Q}}\cup\left\{ \alpha_{1},\alpha_{2},\dots,\alpha_{\tau^{*}}\right\} $.
It then follows that every product $\Gamma^{\left(u\right)}\left(q\right)\Gamma^{\left(v\right)}\left(q\right)$
(with $0\leq u,v\le N$) is contained in the $\overline{\mathbb{Q}}$-vector
space\foreignlanguage{american}{
\[
\mathcal{U}_{\tau^{*}}=\textrm{span}_{\overline{\mathbb{Q}}}\left\{ 1,\alpha_{1},\alpha_{2},\dots,\alpha_{\tau^{*}},\alpha_{i}\alpha_{i^{\prime}}:1\leq i\leq i^{\prime}\leq\tau^{*}\right\} ,
\]
implying} $c_{j,q}\in\mathcal{U}_{\tau^{*}}$.

Since the generating set for $\mathcal{U}_{\tau^{*}}$ has cardinality
\[
1+\tau^{*}+\frac{\tau^{*}\left(\tau^{*}+1\right)}{2}=\frac{\left(\tau^{*}+1\right)\left(\tau^{*}+2\right)}{2}\leq\frac{\left(\tau+2\right)\left(\tau+3\right)}{2},
\]
one can see that 
\[
\dim_{\overline{\mathbb{Q}}}\left(\mathcal{U}_{\tau^{*}}\right)\le\frac{\left(\tau+2\right)\left(\tau+3\right)}{2}.
\]
However, we also know (from arguments similar to those employed in
the proof of Theorem 1) that the elements $\left\{ c_{0,q},c_{1,q},\dots,c_{\left\lfloor N/2\right\rfloor ,q}\right\} $
are $\overline{\mathbb{Q}}$-linearly independent, forcing the span
of this set to have dimension $\left\lfloor N/2\right\rfloor +1$.
Given that this span is contained in $\mathcal{U}_{\tau^{*}}$, it
follows that 
\[
\left\lfloor N/2\right\rfloor +1\le\dim_{\overline{\mathbb{Q}}}\left(\mathcal{U}_{\tau^{*}}\right)\le\frac{\left(\tau+2\right)\left(\tau+3\right)}{2}
\]
\[
\Longrightarrow2\left\lfloor N/2\right\rfloor +2\le\left(\tau+2\right)\left(\tau+3\right)
\]
\[
\Longrightarrow\tau\ge\sqrt{2\left\lfloor N/2\right\rfloor +9/4}-5/2.
\]
Thus,
\[
N\delta_{q}\left(N\right)\ge\max\left\{ 0,\left\lceil \sqrt{2\left\lfloor N/2\right\rfloor +9/4}-5/2\right\rceil \right\} \geq\max\left\{ 0,\sqrt{N}-5/2\right\} 
\]
\[
\Longrightarrow\delta_{q}\left(N\right)\geq\beta\left(N\right)=\dfrac{\max\left\{ 0,\sqrt{N}-5/2\right\} }{N},
\]
where $\beta\left(N\right)\asymp N^{-1/2}$ as $N\rightarrow\infty$.
\end{proof}
\begin{singlespace}
Although the indicated lower bound ($\beta\left(N\right)$) is unaffected
by $q$, this does not imply that $\delta_{q}\left(N\right)$ itself
is constant over $q\in\frac{1}{2}\mathbb{Z}\setminus\mathbb{Z}_{\leq0}$.
After all, $\beta\left(N\right)\asymp N^{-1/2}$ converges to $0$
in the limit, thus providing a rather weak (and not very informative)
lower bound.\footnote{It may be of interest to note that $\beta\left(N\right)$ is larger
in magnitude than the asymptotic lower bound for the density of irrational
$\zeta\left(n\right)$ among $n\in\left\{ 3,5,\ldots,N\right\} $
given by Fischler, Sprang, and Zudilin {[}9{]}:
\[
\beta_{\zeta}\left(N\right)\asymp\dfrac{2^{\left(1-\varepsilon\right)\ln\left(N\right)/\ln\left(\ln\left(N\right)\right)}}{\left(N-1\right)/2}\asymp N^{\left(1-\varepsilon\right)\ln\left(2\right)/\ln\left(\ln\left(N\right)\right)-1},
\]
for arbitrarily small $\varepsilon>0$.}
\end{singlespace}

\section{Values of $q\in\mathbb{Q}\setminus\tfrac{1}{2}\mathbb{Z}$}

\begin{singlespace}
\noindent In this section, we consider transcendental elements in
sequences $\left\{ \Gamma^{\left(n\right)}\left(q\right)\right\} _{n\geq1}$
for $q\in\mathbb{Q}\setminus\tfrac{1}{2}\mathbb{Z}$. Although the
approach is similar to that of Section 2.1, it is restricted by the
fact that Euler's reflection formula\foreignlanguage{american}{ (in
conjunction with the functional equation)} does not yield a simple
closed form for $\Gamma\left(q+t\right)\Gamma\left(q-t\right)$ in
this case. Instead, the reflection formula naturally couples the points
$q$ and $1-q$, leading to a disjunctive result for the sequences
$\left\{ \Gamma^{\left(n\right)}\left(q\right)\right\} _{n\geq1}$
and $\left\{ \Gamma^{\left(n\right)}\left(1-q\right)\right\} _{n\geq1}$.
\end{singlespace}
\begin{thm}
For $q\in\mathbb{Q}\setminus\tfrac{1}{2}\mathbb{Z}$, at least one
of the sequences $\left\{ \Gamma^{\left(n\right)}\left(q\right)\right\} _{n\geq1}$,
$\left\{ \Gamma^{\left(n\right)}\left(1-q\right)\right\} _{n\geq1}$
contains infinitely many transcendental elements.
\end{thm}
\begin{proof}
\begin{singlespace}
\phantom{}
\end{singlespace}

\begin{singlespace}
\medskip{}
\noindent Let\foreignlanguage{american}{ $q\in\mathbb{Q}\setminus\frac{1}{2}\mathbb{Z}$,
and note that Euler's reflection formula gives
\[
\Gamma\left(q+t\right)\Gamma\left(1-q-t\right)=\dfrac{\pi}{\sin\left(\pi\left(q+t\right)\right)}
\]
\[
={\displaystyle \sum_{j=0}^{\infty}}d_{j,q}\pi^{j+1}t^{j},
\]
with $d_{j,q}\in\overline{\mathbb{Q}}$ for all $j\geq0$ (because
$q\in\mathbb{Q}$ implies $\sin\left(\pi q\right),\cos\left(\pi q\right)\in\overline{\mathbb{Q}}$
and each $t$-derivative introduces exactly one factor of $\pi$ by
the chain rule) and $d_{j,q}\neq0$ for infinitely many $j$ (because
$\pi/\sin\left(\pi\left(q+t\right)\right)$ is not a polynomial in
$t$). Moreover, from the Taylor-series expansions of $\Gamma\left(q+t\right)$
and $\Gamma\left(1-q-t\right)$ at $t=0$, we see
\[
\Gamma\left(q+t\right)\Gamma\left(1-q-t\right)=\left({\displaystyle \sum_{\ell_{1}=0}^{\infty}}\Gamma^{\left(\ell_{1}\right)}\left(q\right)\frac{t^{\ell_{1}}}{\ell_{1}!}\right)\left({\displaystyle \sum_{\ell_{2}=0}^{\infty}}\Gamma^{\left(\ell_{2}\right)}\left(1-q\right)\frac{\left(-t\right)^{\ell_{2}}}{\ell_{2}!}\right)
\]
\[
={\displaystyle \sum_{j=0}^{\infty}}\left({\displaystyle \sum_{i=0}^{j}}\frac{\left(-1\right)^{j-i}\Gamma^{\left(i\right)}\left(q\right)\Gamma^{\left(j-i\right)}\left(1-q\right)}{i!\left(j-i\right)!}\right)t^{j},
\]
implying
\begin{equation}
d_{j,q}\pi^{j+1}={\displaystyle \sum_{i=0}^{j}}\frac{\left(-1\right)^{j-i}\Gamma^{\left(i\right)}\left(q\right)\Gamma^{\left(j-i\right)}\left(1-q\right)}{i!\left(j-i\right)!}
\end{equation}
for all $j\geq0$.}
\end{singlespace}

\selectlanguage{american}%
\begin{singlespace}
Now suppose (for purposes of contradiction) that both sequences $\left\{ \Gamma^{\left(n\right)}\left(q\right)\right\} _{n\geq1}$
and $\left\{ \Gamma^{\left(n\right)}\left(1-q\right)\right\} _{n\geq1}$
contain only finitely many transcendental elements, and choose $n^{*}\in\mathbb{Z}_{\geq1}$
such that both $\Gamma^{\left(n\right)}\left(q\right)\in\overline{\mathbb{Q}}$
and $\Gamma^{\left(n\right)}\left(1-q\right)\in\overline{\mathbb{Q}}$
for all $n>n^{*}$, with
\[
\mathcal{W}=\textrm{span}_{\overline{\mathbb{Q}}}\left\{ 1,\Gamma\left(q\right),\Gamma\left(1-q\right),\Gamma^{\left(1\right)}\left(q\right),\Gamma^{\left(1\right)}\left(1-q\right),\Gamma^{\left(2\right)}\left(q\right),\Gamma^{\left(2\right)}\left(1-q\right),\ldots,\Gamma^{\left(n^{*}\right)}\left(q\right),\Gamma^{\left(n^{*}\right)}\left(1-q\right)\right\} .
\]
For any $j>2n^{*}$, each summand in (9) has at least one algebraic
factor (because $i+\left(j-i\right)=j>2n^{*}$), and therefore lies
in $\mathcal{W}$. Thus, $d_{j,q}\pi^{j+1}\in\mathcal{W}$ for all
$j>2n^{*}$. However, since $d_{j,q}\neq0$ for infinitely many $j$,
it follows that $\mathcal{W}$ contains infinitely many distinct powers
of $\pi$, forcing a non-trivial polynomial relation over $\overline{\mathbb{Q}}$
that is satisfied by $\pi$. This contradicts the transcendence of
$\pi$, thereby confirming that at least one of the sequences $\left\{ \Gamma^{\left(n\right)}\left(q\right)\right\} _{n\geq1}$,
$\left\{ \Gamma^{\left(n\right)}\left(1-q\right)\right\} _{n\geq1}$
contains infinitely many transcendental elements.
\end{singlespace}
\end{proof}

\section{Conclusion}

\begin{singlespace}
\noindent In the present study, we investigated the ``infinitely
often'' transcendence of $\Gamma^{\left(n\right)}\left(q\right)$
for $q\in\mathbb{Q}\setminus\mathbb{Z}_{\leq0}$. Theorem 1 showed
that for all $q\in\tfrac{1}{2}\mathbb{Z}\setminus\mathbb{Z}_{\leq0}$,
the sequence $\left\{ \Gamma^{\left(n\right)}\left(q\right)\right\} _{n\geq1}$
contains infinitely many transcendental elements, and Theorem 2 gave
a lower bound, $\beta\left(N\right)\asymp N^{-1/2}$, for the density
of transcendental elements $\Gamma^{\left(n\right)}\left(q\right)$
among $n\in\left\{ 1,2,\ldots,N\right\} $. Although it was not possible
to extend the full implications of Theorem 1 to $q\in\mathbb{Q}\setminus\tfrac{1}{2}\mathbb{Z}$,
we were able to provide a weaker (disjunctive) result, showing that
at least one of the sequences $\left\{ \Gamma^{\left(n\right)}\left(q\right)\right\} _{n\geq1}$,
$\left\{ \Gamma^{\left(n\right)}\left(1-q\right)\right\} _{n\geq1}$
contains infinitely many transcendental elements.
\end{singlespace}

Natural directions for further research include: (i) improving Theorem
3 by removing the disjunctive limitation; and (ii) for fixed $n\in\mathbb{Z}_{\geq1}$
and $\kappa\in\left[0,1\right)\cap\mathbb{Q}$, ascertaining lower
bounds on the densities of transcendental elements $\Gamma^{\left(n\right)}\left(m+\kappa\right)$
among $\left|m\right|\in\left\{ 0,1,\ldots,M\right\} $. Progress
toward the latter objective is presented in Powers {[}14{]}.\medskip{}
\medskip{}

\end{document}